\documentclass[12pt,a4paper]{article}
\usepackage[utf8]{inputenc}
\usepackage[english]{babel}
\usepackage{amsmath}
\usepackage{amsfonts}
\usepackage{amssymb}
\usepackage[left=2cm,right=2cm,top=2cm,bottom=2cm]{geometry}
\usepackage{amsthm}
\usepackage{multicol} 
\usepackage{latexsym}

\newcommand{\bs}[1]{\boldsymbol{#1}}
\newcommand{\Div}[1]{\text{div}\,}

\newtheorem{lemma}{{Lemma}}[section]
\newtheorem{corollary}{Corollary}
\newtheorem{theorem}{{Theorem}}[section]

\numberwithin{equation}{section}
\title{Analysis of a Boundary-Domain Integral Equation System for the Mixed Interior Diffusion BVP with Variable Coefficient Based on a New Family of Parametrices}

\begin{document}\maketitle
\begin{center}
 S.E. MIKHAILOV, C.F.PORTILLO\footnote{Corresponding author, c.portillo@brookes.ac.uk}
\end{center}
\begin{abstract}
A mixed boundary value problem for the diffusion equation in non-homogeneous media partial differential equation is reduced to a system of direct segregated parametrix-based Boundary-Domain Integral Equations (BDIEs). We use a parametrix different from the one employed by Mikhailov in \cite{localised,carlos2} and Chkadua, Mikhailov, Natroshvili in \cite{mikhailov1}. We prove the equivalence between the original BVP and the corresponding BDIE system. The invertibility and Fredholm properties of the boundary-domain integral operators are also analysed.
\end{abstract}

\section{Introduction}
Boundary Domain Integral Equation Systems (BDIES) are often derived from a wide class of boundary value problems with variable coefficient in domains with smooth boundary: cf. \cite{mikhailov1} for a scalar mixed elliptic BVP in bounded domains; cf. \cite{exterior} for the corresponding problem in unbounded domains; cf. \cite{carlos1} for the mixed problem for the incompressible Stokes system in bounded domains, and cf. \cite{dufera}  for a 2D mixed elliptic problem in bounded domains. Further results on the theory of BDIES derived from BVP with variable coefficient can be found on
\cite{carlos2, iterative,localised,united, traces,numerics,miksolandreg,dufera,ayele}. Let us note that these type of BVPs model, for example, the heat transfer in non homogeneous media or the motion of a laminar fluid with variable viscosity. 
 
In order to deduce a BDIES from a BVP with variable coefficient a parametrix (see formula \eqref{parametrixdef}) is required since it keeps a strong relationship with the corresponding fundamental solution of the analogous BVP with constant coefficient. Using this relationship, it is possible to stablish further relations between the surface and volume potential type operators of the variable coefficient case with their counterparts from the constant coefficient case, see, e.g. \cite[Formulae (3.10)-(3.13)]{mikhailov1}, \cite[Formulae (34.10)-(34.16)]{carlos1}. 

For this scalar equation, a family of weakly singular parametrices of the form $P^{y}(x,y)$ for the particular operator
\[\mathcal{A}u(x) := \sum_{i=1}^{3}\dfrac{\partial}{\partial x_{i}}\left(a(x)\dfrac{\partial u(x)}{\partial x_{i}}\right),\]
 has been studied in \cite{mikhailov1,miksolandreg,exterior}. Note that the superscript in $P^{y}(x,y)$ means that $P^{y}(x,y)$ is a function of the variable coefficient depending on $y$, this is 
 $$P^{y}(x,y)= P(x,y;a(y))=\dfrac{-1}{4\pi a(y)\vert x-y\vert}.$$
 
 \textit{In this paper,} we study the parametrices for the operator $\mathcal{A}$ of the form
  $$P^{x}(x,y)= P(x,y;a(x))=\dfrac{-1}{4\pi a(x)\vert x-y\vert}.$$
 which can be useful at the time of studying BDIES derived from a BVP with a system of PDEs as illustrated in the following example. 
Let $\bs{\mathcal{B}}$ be the compressible Stokes operator with variable viscosity:
 \begin{align}
\mathcal{B}_{j}(p, \boldsymbol{v})(x ):&= 
\frac{\partial}{\partial x_{i}}\left(\mu(x)\nonumber
\left(\frac{\partial v_{j}}{\partial x_{i}} + \frac{\partial v_{i}}{\partial x_{j}}
-\frac{2}{3}\delta_{i}^{j} \Div (\boldsymbol{v})\right)\right)
-\frac{\partial p}{\partial x_{j}},\\
&j,i\in \lbrace 1,2,3\rbrace.\nonumber
\end{align}
Following the same notation as for $P^{y}(x,y)$, here $\bs{P}^{x,y}_{\mathcal{B}}(x,y)$ mean that the parametrix for the operator $\mathcal{B}$  includes the variable coefficient $\mu$ depending on $x$ and \textit{also} includes $\mu$ depending on $y$. This is $\bs{P}_{\mathcal{B}}^{x,y}(x,y):=(q^{k}(x,y),u_{j}^{k}(x,y))$
\begin{align*}
q^{k}(x,y)&=\frac{\mu(x)}{\mu(y)}\frac{(x_{k}-y_{k})}{4\pi\vert x-y\vert^{3}},\\
u_{j}^{k}(x,y)&=-\frac{1}{8\pi\mu(y)}\left\lbrace \dfrac{\delta_{j}^{k}}{\vert x-y\vert}+\dfrac{(x_{j}-y_{j})(x_{k}-y_{k})}{\vert x-y\vert^{3}}\right\rbrace, \hspace{0.5em}j,k\in \lbrace 1,2,3\rbrace.
\end{align*}

Then, it seems reasonable to study parametrices for a rather more simple problem of the type $P^{x}(x,y;a(x))$, which have not been analysed yet, before embarking in the analysis of boundary domain integral equations for the operator $\mathcal{B}$.  

There are fast computational techniques developed to solve BDIES. One is the discretisation of the BDIES using localised parametrices which leads to systems of linear equations with large sparsed matrices whose solution can be solve by fast iterative methods, see \cite[Section 1 and Section 5]{localised}, see also \cite{localised2,sladek}. Another fast method is using a collocation method along with hierarchical matrix compression technique
in conjunction with the adaptive cross approximation  procedure, this is shown in \cite{numerics}.

In order to study the possible numerical advantages of 
the new family of parametrices of the form $P^{x}(x,y;a(x))$ with respect to the parametrices already studied, it is necessary to prove the unique-solvency of an analogous BDIES derived with this new family of parametrices. This is the main purpose of this paper along with showing useful arguments that can be helpful at the time of studying BDIES derived from BVPs with variable coefficient which use parametrices of the same family studied in here.  

The main differences between the different families of parametrices are the relations between the parametrix-based potentials with their counterparts for the constant coefficient case.  Notwithstanding, the same mapping properties in Sobolev-Bessel potential spaces still hold allowing us to prove the equivalence between the BDIES and the BVP. 

An analysis of the uniqueness of the BDIES is performed by studying the Fredholm properties of the matrix operator which defines the system.

\section{Preliminaries and the BVP}
Let $\Omega=\Omega^{+}$ be a bounded simply connected domain, $\Omega^{-}:=\mathbb{R}^{3}\smallsetminus\bar{\Omega}^{+}$ the complementary (unbounded) subset of $\Omega$. The boundary $S := \partial\Omega$ is simply connected, closed and infinitely differentiable, $S\in\mathcal{C}^{\infty}$. Furthermore, $S :=\overline{S}_{N}\cup \overline{S}_{D}$ where both $S_{N}$ and $S_{D}$ are non-empty, connected disjoint manifolds of $S$. The border of these two submanifolds is also infinitely differentiable, $\partial S_{N}= \partial S_{D}\in\mathcal{C}^{\infty}$. 

Let us introduce the following partial differential equation with variable smooth positive coefficient $a(x)\in \mathcal{C}^{\infty}(\overline{\Omega})$:
\begin{equation}\label{ch4operatorA}
\mathcal{A}u(x):=\sum_{i=1}^{3}\dfrac{\partial}{\partial x_{i}}\left(a(x)\dfrac{\partial u(x)}{\partial x_{i}}\right)=f(x),\,\,x\in \Omega,
\end{equation}
 where $u(x)$ is an unknown function and $f$ is a given function on $\Omega$. It is easy to see that if $a\equiv 1$ then, the operator $\mathcal{A}$ becomes $\Delta$, the Laplace operator.
 
 We will use the following function spaces in this paper (see e.g. \cite{mclean, lions, hsiao} for more details). Let $\mathcal{D}'(\Omega)$ be the
Schwartz distribution space; $H^{s}(\Omega)$ and $H^{s}(S)$ with $s\in \mathbb{R}$,
the Bessel potential spaces; the space $H^{s}_{K}(\mathbb{R}^{3})$ consisting of all the distributions of $H^{s}(\mathbb{R}^{3})$ whose support is inside of a compact set $K\subset \mathbb{R}^{3}$; the spaces consisting of distributions in $H^{s}(K)$ for every compact $K\subset \overline{\Omega^{-}},\hspace{0.1em}s\in\mathbb{R}$. Let us introduce the following Sobolev-Bessel potentials on the boundary:
\begin{equation*}
\widetilde{H}^{s}(S):= \overline{\mathcal{C}^{\infty}_{0}(S)}^{\parallel\cdot\parallel_{H^{s}(\mathbb{R}^{3})}},\quad\quad\quad H^{s}(S):= \overline{\mathcal{C}^{\infty}_{0}(S)}^{\parallel\cdot\parallel_{H^{s}(S)}},
\end{equation*} 
 whose characterizations are given as follows: $\widetilde{H}^{s}(S_{1})=\lbrace g\in H^{s}(S): {\rm supp}(g)\subset \overline{S_{1}}\rbrace$; $H^{s}(S_{1})=\lbrace g\vert_{S_{1}}: g\in H^{s}(S)\rbrace$, where the notation $g\vert_{S_{1}} = r_{S_{1}}g$ is used to indicate the restriction of the function $g$ from $S$ to $S_{1}$.  
 
We will make use of the space, see e.g. \cite{costabel, mikhailov1},
\begin{center}
 $H^{1,0}(\Omega; \mathcal{A}):= \lbrace u \in H^{1}(\Omega): \mathcal{A}u\in L^{2}(\Omega)\rbrace $
 \end{center} 
 which is a Hilbert space with the norm defined by
\begin{center}
$\parallel u \parallel^{2}_{H^{1,0}(\Omega; \mathcal{A})}:=\parallel u \parallel^{2}_{H^{1}(\Omega)}+\parallel \mathcal{A}u  \parallel^{2}_{L^{2}(\Omega)}$.
\end{center}

\paragraph{Traces and conormal derivatives.}
For a scalar function $w\in H^{s}(\Omega^\pm)$, $s>1/2$, the trace operator $\gamma^{\pm}(\,\cdot \,):=\gamma_{S}^\pm(\,\cdot\,)$, acting on $w$ is well defined and $\gamma^{\pm}w\in H^{s-\frac{1}{2}}(S)$ (see, e.g., \cite{mclean, traces}). 
For $u\in H^{s}(\Omega)$, $s>3/2$, we can define on $S$ the  conormal derivative operator, $T^{\pm}$, in the classical (trace) sense
\begin{equation*}\label{ch4conormal}
T^{\pm}_{x}u := \sum_{i=1}^{3}a(x)\gamma^{\pm}\left( \dfrac{\partial u}{\partial x_{i}}\right)^{\pm}n_{i}^{\pm}(x),
\end{equation*}
where $n^{+}(x)$ is the exterior unit normal vector directed \textit{outwards} the interior domain $\Omega$ at a point $x\in S$. Similarly, $n^{-}(x)$ is the unit normal vector directed \textit{inwards} the interior domain $\Omega$ at a point $x\in S$. 

Furthermore, we will use the notation $T^{\pm}_{x}u$ or $T^{\pm}_{y}u$ to emphasise which respect to which variable we are differentiating. When the variable of differentiation is obvious or is a dummy variable, we will simply use the notation $T^{\pm}u$.

Moreover, for any function $u\in H^{1,0}(\Omega; \mathcal{A})$, the \textit{canonical} conormal derivative $T^{\pm}u\in H^{-\frac{1}{2}}(\Omega)$, is well defined, cf. \cite{costabel,mclean,traces},
\begin{equation}\label{ch4green1}
\langle T^{\pm}u, w\rangle_{S}:= \pm \int_{\Omega^{\pm}}[(\gamma^{-1}\omega)\mathcal{A}u +E(u,\gamma^{-1}w)] dx,\,\, w\in H^{\frac{1}{2}}(S),
\end{equation}
where $\gamma^{-1}: H^{\frac{1}{2}}(S)\longrightarrow H_{K}^{1}(\mathbb{R}^{3})$ is a continuous right inverse to the trace operator whereas the function $E$ is defined as
\begin{equation*}\label{ch4functionalE}
E(u,v)(x):=\sum_{i=1}^{3}a(x)\dfrac{\partial u(x)}{\partial x_{i}}\dfrac{\partial v(x)}{\partial x_{i}},
\end{equation*}
and $\langle\, \cdot \, , \, \cdot \, \rangle_{S}$ represents the $L^{2}-$based dual form on $S$.

We aim to derive boundary-domain integral equation systems for the following \textit{mixed} boundary value problem. Given $f\in L^{2}(\Omega)$, $\phi_{0}\in H^{\frac{1}{2}}(S_D)$ and $\psi_{0}\in H^{-\frac{1}{2}}(S_N)$, we seek a function $u\in H^{1}(\Omega)$ such that 
\begin{subequations}\label{ch4BVP}
\begin{align}
\mathcal{A}u&=f,\hspace{1em}\text{in}\hspace{1em}\Omega\label{ch4BVP1};\\
r_{S_{D}}\gamma^{+}u &= \phi_{0},\hspace{1em}\text{on}\hspace{1em} S_{D}\label{ch4BVP2};\\
r_{S_{N}}T^{+}u &=\psi_{0},\hspace{1em}\text{on}\hspace{1em} S_{N};\label{ch4BVP3}
\end{align}
where equation \eqref{ch4BVP1} is understood in the weak sense, the Dirichlet condition \eqref{ch4BVP2} is understood in the trace sense and the Neumann condition \eqref{ch4BVP3} is understood in the functional sense \eqref{ch4green1}.
\end{subequations}

By Lemma 3.4 of \cite{costabel} (cf. also Theorem 3.9 in \cite{traces}), the first Green identity holds for any $u\in H^{1,0}(\Omega; \mathcal{A})$ and $v\in H^{1}(\Omega)$,
\begin{equation}\label{ch4green1.1}
\langle T^{\pm}u, \gamma^{+}v\rangle_{S}:= \pm \int_{\Omega}[v\mathcal{A}u +E(u,v)] dx.
\end{equation}
The following assertion is well known and can be proved, e.g., using the Lax-Milgram lemma as in \cite[Chapter 4]{steinbach}. 
\begin{theorem}\label{ch4Thomsol}
The boundary value problem \eqref{ch4BVP} has one and only one  solution.
\end{theorem}
 
\section{Parametrices and remainders}
 We define a parametrix (Levi function) $P(x,y)$ for a differential operator $\mathcal{A}_{x}$ differentiating with respect to $x$ as a function on two variables that satisfies
 \begin{equation}\label{parametrixdef}
 \mathcal{A}_{x}P(x,y) = \delta(x-y)+R(x,y).
 \end{equation}
For a given operator $\mathcal{A}$, the parametrix is not unique. For example, the parametrix 
\begin{equation*}\label{ch4P2002}
P^y(x,y)=\dfrac{1}{a(y)} P_\Delta(x-y),\hspace{1em}x,y \in \mathbb{R}^{3},
\end{equation*} 
was employed in \cite{localised, mikhailov1}, for the operator $\mathcal{A}$ defined in \eqref{ch4operatorA},  where
\begin{equation*}\label{ch4fundsol}
P_{\Delta}(x-y) = \dfrac{-1}{4\pi \vert x-y\vert}
\end{equation*}
is the fundamental solution of the Laplace operator.
The remainder corresponding to the parametrix $P^{y}$ is 
\begin{equation*}
\label{ch43.4} R^y(x,y)
=\sum\limits_{i=1}^{3}\frac{1}{a(y)}\, \frac{\partial a(x)}{\partial x_i} \frac{\partial }{\partial x_i}P_\Delta(x-y)
\,,\;\;\;x,y\in {\mathbb R}^3.
\end{equation*}
{\em In this paper}, for the same operator $\mathcal{A}$ defined in \eqref{ch4operatorA}, we will use another parametrix,  
\begin{align}\label{ch4Px}
P(x,y):=P^x(x,y)=\dfrac{1}{a(x)} P_\Delta(x-y),\hspace{1em}x,y \in \mathbb{R}^{3},
\end{align}
which leads to the corresponding remainder 
\begin{align*}\label{ch4remainder}
R(x,y) =R^x(x,y) &= 
-\sum\limits_{i=1}^{3}\dfrac{\partial}{\partial x_{i}}
\left(\frac{1}{a(x)}\dfrac{\partial a(x)}{\partial x_{i}}P_{\Delta}(x,y)\right)\\
& =
-\sum\limits_{i=1}^{3}\dfrac{\partial}{\partial x_{i}}
\left(\dfrac{\partial \ln a(x)}{\partial x_{i}}P_{\Delta}(x,y)\right),\hspace{0.5em}x,y \in \mathbb{R}^{3}.
\end{align*}

Note that the both remainders $R_x$ and $R_y$ are weakly singular, i.e., \[
R^x(x,y),\,R^y(x,y)\in \mathcal{O}(\vert x-y\vert^{-2}).\]
 This is due to the smoothness of the variable coefficient $a$.

\section{Volume and surface potentials}

The volume parametrix-based Newton-type potential and the remainder potential are respectively defined, for $y\in\mathbb R^3$, as
\begin{align*}
\mathcal{P}\rho(y)&:=\displaystyle\int_{\Omega} P(x,y)\rho(x)\hspace{0.25em}dx\\
\mathcal{R}\rho(y)&:=\displaystyle\int_{\Omega} R(x,y)\rho(x)\hspace{0.25em}dx.
\end{align*}

 The parametrix-based single layer and double layer  surface potentials are defined for $y\in\mathbb R^3:y\notin S $, as 
\begin{equation*}\label{ch4SL}
V\rho(y):=-\int_{S} P(x,y)\rho(x)\hspace{0.25em}dS(x),
\end{equation*}
\begin{equation*}\label{ch4DL}
W\rho(y):=-\int_{S} T_{x}^{+}P(x,y)\rho(x)\hspace{0.25em}dS(x).
\end{equation*}

We also define the following pseudo-differential operators associated with direct values of the single and  double layer potentials and with their conormal derivatives, for $y\in S$,
\begin{align}
\mathcal{V}\rho(y)&:=-\int_{S} P(x,y)\rho(x)\hspace{0.25em}dS(x),\nonumber \\
\mathcal{W}\rho(y)&:=-\int_{S} T_{x}P(x,y)\rho(x)\hspace{0.25em}dS(x),\nonumber\\
\mathcal{W'}\rho(y)&:=-\int_{S} T_{y}P(x,y)\rho(x)\hspace{0.25em}dS(x),\nonumber\\
\mathcal{L}^{\pm}\rho(y)&:=T_{y}^{\pm}{W}\rho(y)\nonumber.
\end{align}

The operators $\mathcal P, \mathcal R, V, W, \mathcal{V}, \mathcal{W}, \mathcal{W'}$ and $\mathcal{L}$ can be expressed in terms the volume and surface potentials and operators associated with the Laplace operator, as follows
\begin{align}
\mathcal{P}\rho&=\mathcal{P}_{\Delta}\left(\dfrac{\rho}{a}\right),\label{ch4relP}\\
\mathcal{R}\rho&=\nabla\cdot\left[\mathcal{P}_{\Delta}(\rho\,\nabla \ln a)\right]-\mathcal{P}_{\Delta}(\rho\,\Delta \ln a),\label{ch4relR}\\
V\rho &= V_{\Delta}\left(\dfrac{\rho}{a}\right),\label{ch4relSL}\\
\mathcal{V}\rho &= \mathcal{V}_{\Delta} \left( \dfrac{\rho}{a}\right),\label{ch4relDVSL}\\
W\rho &= W_{\Delta}\rho -V_{\Delta}\left(\rho\frac{\partial \ln a}{\partial n}\right),\label{ch4relDL}\\
\mathcal{W}\rho &= \mathcal{W}_{\Delta}\rho -\mathcal{V}_{\Delta}\left(\rho\frac{\partial \ln a}{\partial n}\right),\label{ch4relDVDL}\\
\mathcal{W}'\rho &= a \mathcal{W'}_{\Delta}\left(\dfrac{\rho}{a}\right),\label{ch4relTSL} \\
\mathcal{L}^{\pm}\rho &= \widehat{\mathcal{L}}\rho - aT^{\pm}_\Delta V_{\Delta}\left(\rho\frac{\partial \ln a}{\partial n}\right),
\label{ch4relTDL}\\
\widehat{\mathcal{L}}\rho &:= a\mathcal{L}_{\Delta}\rho.\label{ch4hatL}
\end{align}

The symbols with the subscript $\Delta$ denote the analogous operator for the constant coefficient case, $a\equiv 1$. Furthermore, by the Liapunov-Tauber theorem, $\mathcal{L}_{\Delta}^{+}\rho = \mathcal{L}_{\Delta}^{-}\rho = \mathcal{L}_{\Delta}\rho$.

Using relations \eqref{ch4relP}-\eqref{ch4hatL} it is now rather simple to obtain, similar to \cite{mikhailov1}, the mapping properties, jump relations and invertibility results for the parametrix-based  surface and volume potentials, provided in theorems/corollary \ref{ch4thmUR}-\ref{ch4thinvV}, from the well-known properties of their constant-coefficient counterparts (associated with the Laplace equation).
\newpage
 
 \begin{theorem}\label{ch4thmUR} Let $s\in \mathbb{R}$. Then, the following operators are continuous,
 \begin{equation*}
 \mathcal{P}:\widetilde{H}^{s}(\Omega) \longrightarrow H^{s+2}(\Omega),\hspace{0.5em} s\in \mathbb{R},\label{ch4mpvp1}\\
 \end{equation*}
 \begin{equation*}
 \mathcal{P}: H^{s}(\Omega) \longrightarrow H^{s+2}(\Omega),\hspace{0.5em} s>-\dfrac{1}{2},\label{ch4mpvp2}\\
 \end{equation*}
 \begin{equation*}
 \mathcal{R}:\widetilde{H}^{s}(\Omega) \longrightarrow H^{s+1}(\Omega),\hspace{0.5em} s\in \mathbb{R},\label{ch4mpvp3}\\
 \end{equation*}
 \begin{equation*}
 \mathcal{R}: H^{s}(\Omega) \longrightarrow H^{s+1}(\Omega),\hspace{0.5em} s>-\dfrac{1}{2}\,.\label{ch4mpvp4}
 \end{equation*}
 \end{theorem}
 \begin{corollary}\label{ch4corcompact}Let $s>\frac{1}{2}$,  let $S_{1}$ be a non-empty submanifold of $S$ with smooth boundary. Then, the following operators are compact:
 \begin{align*}
 \mathcal{R}&: H^{s}(\Omega) \longrightarrow H^{s}(\Omega),\\
 r_{S_{1}}\gamma^{+}\mathcal{R}&: H^{s}(\Omega) \longrightarrow H^{s-\frac{1}{2}}(S_{1}),\\
 r_{S_{1}}T^{+}\mathcal{R}&: H^{s}(\Omega) \longrightarrow H^{s-\frac{3}{2}}(S_{1}).
 \end{align*}
 \end{corollary}   

\begin{theorem}\label{ch4thmappingVW} Let $s\in \mathbb{R}$. Then, the following operators are continuous:
\begin{align*}
V: H^{s}(S) \longrightarrow H^{s+\frac{3}{2}}(\Omega),\\
W: H^{s}(S) \longrightarrow H^{s+\frac{1}{2}}(\Omega).
\end{align*}
\end{theorem}

\begin{theorem} \label{ch4thmappingDVVW}Let $s\in \mathbb{R}$. Then, the following operators are continuous:
\begin{align*}
\mathcal{V}&: H^{s}(S) \longrightarrow H^{s+1}(S),\\
\mathcal{W}&: H^{s}(S) \longrightarrow H^{s+1}(S),\\
\mathcal{W'}&: H^{s}(S) \longrightarrow H^{s+1}(S),\\
\mathcal{L}^\pm&: H^{s}(S) \longrightarrow H^{s-1}(S).
\end{align*}
\end{theorem}

\begin{theorem}\label{ch4thjumps} Let  $\rho\in H^{-\frac{1}{2}}(S)$, $\tau\in H^{\frac{1}{2}}(S)$. Then the following operators jump relations hold:
\begin{align*}
\gamma^{\pm}V\rho&=\mathcal{V}\rho,\\
\gamma^{\pm}W\tau&=\mp\dfrac{1}{2}\tau+\mathcal{W}\tau,\\
T^{\pm}V\rho&=\pm\dfrac{1}{2}\rho+\mathcal{W'}\rho.
\end{align*}
\end{theorem}

\begin{theorem}\label{ch4thcompVW} Let $s\in \mathbb{R}$, let $S_{1}$ and $S_{2}$ be two non-empty manifolds with smooth boundaries, $\partial S_{1}$ and $\partial S_{2}$, respectively. Then, the following operators 
\begin{align*}
r_{S_{2}}\mathcal{V}: \widetilde{H}^{s}(S_{1}) \longrightarrow H^{s}(S_{2}),\\
r_{S_{2}}\mathcal{W}: \widetilde{H}^{s}(S_{1}) \longrightarrow H^{s}(S_{2}),\\
r_{S_{2}}\mathcal{W}':\widetilde{H}^{s}(S_{1}) \longrightarrow H^{s}(S_{2}).
\end{align*}
are compact.
\end{theorem}
\begin{theorem}\label{ch4thinvV} Let $S_{1}$ be a non-empty simply connected submanifold of $S$ with infinitely smooth boundary curve, and $0<s<1$. Then, the operators
\begin{align*}
r_{S_{1}}\mathcal{V}&: \widetilde{H}^{s-1}(S_{1}) \longrightarrow H^{s}(S_{1}),\\
\mathcal{V}&: H^{s-1}(S) \longrightarrow H^{s}(S),
\end{align*}
are invertible.
\end{theorem}
\begin{proof}
Relation \eqref{ch4relSL} gives $\mathcal{V}g = \mathcal{V}_{\Delta}g^{*}$, where $g = g^{*}/a$. The invertibility of $\mathcal{V}$ then follows from the invertibility of $\mathcal{V}_{\Delta}$, see references \cite[Theorem 2.4]{costste} and \cite[Theorem 3.5]{miksolandreg}.
\end{proof}

\begin{theorem}\label{ch4thinvL} Let $S_{1}$ be a non-empty simply connected submanifold of $S$ with infinitely smooth boundary curve, and $0<s<1$. Then, the operator
\begin{align*}
r_{S_{1}}\widehat{\mathcal{L}}: \widetilde{H}^{s}(S_{1}) \longrightarrow H^{s-1}(S_{1}),
\end{align*}
is invertible whilst the operators
\begin{align*}
r_{S_{1}}(\mathcal{L}^{\pm}-\widehat{\mathcal{L}}): \widetilde{H}^{s}(S_{1}) \longrightarrow H^{s-1}(S_{1}),
\end{align*}
are compact.
\end{theorem}
\begin{proof}
Relation \eqref{ch4relTDL} gives
\[ \widehat{\mathcal{L}}\rho = \mathcal{L}^{\pm}\rho + aT^{+}_\Delta V_{\Delta}\left(\rho\frac{\partial \ln a}{\partial n}\right) = \mathcal{L}^{\pm}\rho + aT^{-}_\Delta V_{\Delta}\left(\rho\frac{\partial \ln a}{\partial n}\right).   \] 

Take into account $\widehat{\mathcal{L}}\rho := a\mathcal{L}_{\Delta}\rho$ and the invertibility of the operator $\mathcal{L}_{\Delta}$, see references  \cite[Theorem 2.4]{costste} and \cite[Theorem 3.6]{miksolandreg}; we deduce the invertibility of the operator $\widehat{\mathcal{L}}$. To prove the compactness properties, we consider the identity:
\[ \mathcal{L}^{\pm}\rho - \widehat{\mathcal{L}}\rho =  a\left(\mp \dfrac{1}{2}I - \mathcal{W}'_{\Delta}\right)\left(\rho\frac{\partial \ln a}{\partial n}\right).  \]
Since $\rho\in \widetilde{H}^{s}(S_{1})$, due to the mapping properties of the operator $\mathcal{W}'$, $\mathcal{L}^{\pm} - \rho\widehat{\mathcal{L}}\rho\in H^{s}$. Then, immediately follows from the compact embedding $H^{s}(S)\subset H^{s-1}(S)$, that the operators 
\begin{align*}
r_{S_{1}}(\mathcal{L}^{\pm}-\widehat{\mathcal{L}}): \widetilde{H}^{s}(S_{1}) \longrightarrow H^{s-1}(S_{1}),
\end{align*}
are compact. 
\end{proof}

\section{Third Green identities and integral relations}

In this section we provide the results similar to the ones in \cite{mikhailov1} but for our, different, parametrix \eqref{ch4Px}.  
  
Let  $u,v\in H^{1,0}(\Omega;\mathcal{A})$.   Subtracting from the first Green identity \eqref{ch4green1.1} its counterpart with the swapped $u$ and $v$, we arrive at the second Green identity, see e.g. \cite{mclean},
\begin{equation}\label{ch4green2}
\displaystyle\int_{\Omega}\left[u\,\mathcal{A}v - v\,\mathcal{A}u\right]dx= \int_{S}\left[u\, T^{+}v \,-\,v\, T^{+}u\,\right]dS(x). 
\end{equation}
Taking now $v(x):=P(x,y)$, we obtain from \eqref{ch4green2} by the standard limiting procedures (cf. \cite{miranda}) the third Green identity  for any function $u\in H^{1,0}(\Omega;\mathcal{A})$:
\begin{equation}\label{ch4green3}
u+\mathcal{R}u-VT^{+}u+W\gamma^{+}u=\mathcal{P}\mathcal{A}u,\hspace{1em}\text{in}\hspace{0.2em}\Omega.
\end{equation}

If $u\in H^{1,0}(\Omega; \mathcal{A})$ is a solution of the partial differential equation \eqref{ch4BVP1}, then, from \eqref{ch4green3} we obtain:

\begin{equation}\label{ch43GV}
u+\mathcal{R}u-VT^{+}u+W\gamma^{+}u=\mathcal{P}f,\hspace{0.5em}in\hspace{0.2em}\Omega;
\end{equation}
\begin{equation}\label{ch43GG}
\dfrac{1}{2}\gamma^{+}u+\gamma^{+}\mathcal{R}u-\mathcal{V}T^{+}u+\mathcal{W}\gamma^{+}u=\gamma^{+}\mathcal{P}f,\hspace{0.5em}on\hspace{0.2em}S.
\end{equation}

For some distributions $f$, $\Psi$ and $\Phi$, we consider a more general, indirect integral relation associated with the third Green identity \eqref{ch43GV}:

\begin{equation}\label{ch4G3ind}
u+\mathcal{R}u-V\Psi+W\Phi=\mathcal{P}f,\hspace{0.5em}{\rm in\ }\Omega.
\end{equation}

\begin{lemma}\label{ch4lema1}Let $u\in H^{1}(\Omega)$, $f\in L_{2}(\Omega)$, $\Psi\in H^{-\frac{1}{2}}(S)$ and $\Phi\in H^{\frac{1}{2}}(S)$ satisfying the relation \eqref{ch4G3ind}. Then $u$ belongs to $H^{1,0}(\Omega, \mathcal{A})$; solves the equation $\mathcal{A}u=f$ in $\Omega$, and the following identity is satisfied,
\begin{equation}\label{ch4lema1.0}
V(\Psi- T^{+}u) - W(\Phi- \gamma^{+}u) = 0\hspace{0.5em}\text{in}\hspace{0.5em}\Omega.
\end{equation}
\end{lemma}
\begin{proof}
First, let us prove that $u\in H^{1,0}(\Omega;\mathcal{A})$. Since $u\in H^{1}(\Omega)$ by hypothesis, it suffices to prove that $\mathcal{A}u \in L^{2}(\Omega)$. Let us thus take equation \eqref{ch4G3ind} and apply the relations \eqref{ch4relP}, \eqref{ch4relSL} and \eqref{ch4relDL} to obtain
\begin{align}\label{ch4lema1.1}
u=&\,\mathcal{P}f-\mathcal{R}u+V\Psi-W\Phi\nonumber\\
=&\,\mathcal{P}_{\Delta}\left(\dfrac{f}{a}\right)-\mathcal{R}u+V_{\Delta}\left(\dfrac{\Psi}{a}\right)-W_{\Delta}\Phi
+V_{\Delta}\left(\dfrac{\partial\ln a}{\partial n}\,\Phi\right).
\end{align}
We note that $\mathcal{R}u\in H^{2}(\Omega)$ due to the mapping properties given by Theorem \ref{ch4thmUR}. Moreover,  $V_{\Delta}$ and $W_{\Delta}$ in \eqref{ch4lema1.1}   are harmonic potentials, while $\mathcal{P}_{\Delta}$ is the Newtonian potential for the Laplacian, i.e., $\Delta\mathcal{P}_{\Delta}\left(\dfrac{f}{a}\right)=\dfrac{f}{a}$. Consequently,
$
\Delta u = \dfrac{f}{a}-\Delta\mathcal{R}u\in L^{2}(\Omega).
$ 
Hence, $\mathcal{A}u\in L^{2}(\Omega)$ and thus $u\in H^{1,0}(\Omega;\mathcal{A})$.

Since $u\in H^{1,0}(\Omega;\mathcal{A})$, the third Green identity \eqref{ch43GV} is valid for the function $u$, 
and
we proceed subtracting \eqref{ch4green3} from \eqref{ch4G3ind} to obtain
\begin{equation}\label{ch4lema1.3}
W(\gamma^{+}u-\Phi)-V(T^{+}u-\Psi)=\mathcal{P}(\mathcal{A}u-f).
\end{equation}
Let us apply relations \eqref{ch4relP}, \eqref{ch4relSL} and \eqref{ch4relDL} to \eqref{ch4lema1.3}, and then, apply the Laplace operator to both sides. Hence, we obtain
\begin{equation}\label{ch4lema1.5}
\mathcal{A}u-f=0,
\end{equation}
i.e., $u$ solves \eqref{ch4BVP1}.
Finally, substituting \eqref{ch4lema1.5} into \eqref{ch4lema1.3}, we prove \eqref{ch4lema1.0}.
\end{proof}

\begin{lemma}\label{ch4lemma2}
Let $\Psi^{*}\in H^{-\frac{1}{2}}(S)$. If
\begin{equation}\label{ch4lema2i}
V\Psi^{*}(y) = 0,\hspace{2em}y\in\Omega
\end{equation}
then $\Psi^{*}(y) = 0$.
\end{lemma}
\begin{proof} Taking the trace of \eqref{ch4lema2i}gives:
\begin{center}
$\mathcal{V}\Psi^{*}(y) = \mathcal{V}_{\triangle}\left(\dfrac{\Psi^{*}}{a}\right)(y) = 0, \hspace{2em}y\in\Omega$,
\end{center}
from where the result follows due to the invertibility of the operator $\mathcal{V}_{\triangle}$ (cf. Lemma  \ref{ch4thinvV}). 
\end{proof}


\section{BDIE system for the mixed problem}
We aim to obtain a segregated boundary-domain integral equation system for mixed BVP \eqref{ch4BVP}. To this end, let the functions $\Phi_{0}\in H^{\frac{1}{2}}(S)$ and $\Psi_{0}\in H^{-\frac{1}{2}}(S)$ be respective continuations of the boundary functions 
$\phi_{0}\in H^{\frac{1}{2}}(S_{D})$ and 
$\psi_{0}\in H^{-\frac{1}{2}}(S_{N})$ to the whole $S$.
Let us now represent 
\begin{equation*}\label{ch4gTrepr}
\gamma^{+}u=\Phi_{0} + \phi,\quad\quad 
T^{+}u = \Psi_{0} +\psi,\quad\text{ on }\,\, S, 
\end{equation*}
where  $\phi\in\widetilde{H}^{\frac{1}{2}}(S_{N})$ and $\psi\in\widetilde{H}^{-\frac{1}{2}}(S_{D})$ are unknown boundary functions. 

To obtain one of the possible boundary-domain integral equation systems we employ identity \eqref{ch43GV} in the domain $\Omega$, and identity \eqref{ch43GG} on $S$, substituting there $\gamma^{+}u=\Phi_{0} + \phi$ and $T^{+}u = \Psi_{0} +\psi$ and further considering the unknown functions $\phi$ and $\psi$ as formally independent (segregated) of $u$ in $\Omega$. Consequently, we obtain the following system (M12) of two equations for three unknown functions,
\begin{subequations}
\begin{align}
u+\mathcal{R}u-V\psi+W\phi&=F_{0}\hspace{2em}in\hspace{0.5em}\Omega,\label{ch4SM12v}\\
\dfrac{1}{2}\phi+\gamma^{+}\mathcal{R}u-\mathcal{V}\psi+\mathcal{W}\phi&=\gamma^{+}F_{0}-\Phi_{0}\label{ch4SM12g}\hspace{2em}on\hspace{0.5em}S,
\end{align}
\end{subequations}
where
\begin{equation}\label{ch4F0term}
F_{0}=\mathcal{P}f+V\Psi_{0}-W\Phi_{0}.
\end{equation}

 We remark that $F_{0}$ belongs to the space $H^{1}(\Omega)$ in virtue of the mapping properties of the surface and volume potentials, see Theorems \ref{ch4thmUR} and \ref{ch4thmappingVW}.

The system (M12), given by \eqref{ch4SM12v}-\eqref{ch4SM12g} can be written in matrix notation as 
\begin{equation*}
\mathcal{M}^{12}\mathcal{X}=\mathcal{F}^{12},
\end{equation*}
where $\mathcal{X}$ represents the vector containing the unknowns of the system,
\begin{equation*}
\mathcal{X}=(u,\psi,\phi)^{\top}\in H^{1}(\Omega)\times\widetilde{H}^{-\frac{1}{2}}(S_{D})\times\widetilde{H}^{\frac{1}{2}}(S_{N}),
\end{equation*}
the right hand side vector is \[\mathcal{F}^{12}:= [ F_{0}, \gamma^{+}F_{0} - \Psi_{0} ]^{\top}\in H^{1}(\Omega)\times H^{\frac{1}{2}}(S),\]
and the matrix operator $\mathcal{M}^{12}$ is defined by:
\begin{equation*}
   \mathcal{M}^{12}=
  \left[ {\begin{array}{ccc}
   I+\mathcal{R} & -V & W \\
   \gamma^{+}\mathcal{R} & -\mathcal{V} & \dfrac{1}{2}I + \mathcal{W} 
  \end{array} } \right].
\end{equation*}

We note that the mapping properties of the operators involved in the matrix imply the continuity of the operator
\begin{equation*}
    \mathcal{M}^{12}: H^{1}(\Omega)\times\widetilde{H}^{-\frac{1}{2}}(S_{D})\times\widetilde{H}^{\frac{1}{2}}(S_{N})\longrightarrow H^{1}(\Omega)\times H^{\frac{1}{2}}(S).
\end{equation*}

\begin{theorem}\label{ch4EqTh}
Let $f\in L_{2}(\Omega)$. Let $\Phi_{0}\in H^{\frac{1}{2}}(S)$ and $\Psi_{0}\in H^{-\frac{1}{2}}(S)$ be some fixed extensions of $\phi_{0}\in H^{\frac{1}{2}}(S_{D})$ and $\psi_{0}\in H^{-\frac{1}{2}}(S_{N})$ respectively. 
\begin{enumerate}
\item[i)] If some $u\in H^{1}(\Omega)$ solves the BVP \eqref{ch4BVP}, then the triple $(u, \psi, \phi )^{\top}\in H^{1}(\Omega)\times\widetilde{H}^{-\frac{1}{2}}(S_{D})\times\widetilde{H}^{\frac{1}{2}}(S_{N})$ where
\begin{equation}\label{ch4eqcond}
\phi=\gamma^{+}u-\Phi_{0},\hspace{4em}\psi=T^{+}u-\Psi_{0},\hspace{2em}on\hspace{0.5em}S,
\end{equation}
solves the BDIE system (M12). 

\item[ii)] If a triple $(u, \psi, \phi )^{\top}\in H^{1}(\Omega)\times\widetilde{H}^{-\frac{1}{2}}(S_{D})\times\widetilde{H}^{\frac{1}{2}}(S_{N})$ solves the BDIE system then $u$ solves the BVP and the functions $\psi, \phi$ satisfy \eqref{ch4eqcond}.

\item[iii)] The system (M12) is uniquely solvable. 
\end{enumerate}
\end{theorem}

\begin{proof}
First, let us prove item $i)$. Let $u\in H^{1}(\Omega)$ be a solution of the boundary value problem \eqref{ch4BVP} and let $\phi$, $\psi$ be defined by \eqref{ch4eqcond}.
Then, due to \eqref{ch4BVP2} and \eqref{ch4BVP3}, we have \[(\psi,\phi) \in \widetilde{H}^{-\frac{1}{2}}(S_{D})\times\widetilde{H}^{\frac{1}{2}}(S_{N}).\] 

 Then, it immediately follows from the third Green identities \eqref{ch43GV} and \eqref{ch43GG} that the triple $(u,\phi, \psi)$ solves BDIE system $\mathcal{M}^{12}$.

Let us prove now item $ii)$. Let the triple $(u, \psi,\phi )^{\top}\in H^{1}(\Omega)\times\widetilde{H}^{-\frac{1}{2}}(S_{D})\times\widetilde{H}^{\frac{1}{2}}(S_{N})$ solve the BDIE system. Taking the trace of the equation \eqref{ch4SM12v} and substract it from the equation \eqref{ch4SM12g}, we obtain
\begin{equation}\label{ch4M12a1}
\phi=\gamma^{+}u-\Phi_{0}, \hspace{1em} \text{on}\hspace{0.5em}S.
\end{equation}
This means that the first condition in \eqref{ch4eqcond} is satisfied. 
Now, restricting equation \eqref{ch4M12a1} to $S_{D}$, we observe that $\phi$ vanishes as $supp(\phi)\subset S_{N}$. Hence, $\phi_{0}=\Phi_{0}=\gamma^{+}u$ on $S_{D}$ and consequently, the Dirichlet condition of the BVP \eqref{ch4BVP2} is satisfied. 

We proceed using the Lemma \ref{ch4lema1} in the first equation of the system (M12), \eqref{ch4SM12v}, with $\Psi=\psi + \Psi_{0}$ and $\Phi = \phi + \Phi_{0}$ which implies that $u$ is a solution of the equation \eqref{ch4BVP1} and also the following equality:
\begin{equation*}\label{ch4M12a2}
V(\Psi_{0}+\psi - T^{+}u) - W(\Phi_{0} + \phi -\gamma^{+}u) = 0 \text{ in } \Omega.
\end{equation*}
By virtue of \eqref{ch4M12a1}, the second term of the previous equation* vanishes. Hence,
\begin{equation*}\label{ch4M12a3}
V(\Psi_{0}+\psi - T^{+}u)= 0, \quad \text{ in } \Omega.
\end{equation*}
Now, by virtue of Lemma \ref{ch4lemma2} we obtain
\begin{equation}\label{ch4M12a4}
\Psi_{0} + \psi - T^{+}u = 0,\quad\text{ on } S.
\end{equation}
Since $\psi$ vanishes on $S_{N}$, we can conclude that $\Psi_{0}=\psi_{0}$ on $S_{N}$. Consequently, equation \eqref{ch4M12a4} implies that $u$ satisfies the Neumann condition \eqref{ch4BVP3}. 

Item $iii)$ immediately follows from the uniqueness of the solution of the mixed boundary value problem \ref{ch4Thomsol}.
\end{proof}

\begin{lemma}\label{ch4remark}$(F_{0}, \gamma^{+}F_{0}-\Phi_{0})=0$ if and only if $(f, \Phi_{0},\Psi_{0})=0$
\end{lemma}
\begin{proof}
It is trivial that if $(f, \Phi_{0}, \Psi_{0})=0$ then $(F_{0}, \gamma^{+}F_{0}-\Phi_{0})=0$. Conversely, supposing that $(F_{0}, \gamma^{+}F_{0}-\Phi_{0})=0$, then taking into account equation \eqref{ch4F0term} and applying Lemma \ref{ch4lema1} with $F_{0}=0$ as $u$, we deduce that $f=0$ and $V\Psi_{0}-W\Phi_{0}=0$ in $\Omega$. Now, the second equality, $\gamma^{+}F_{0}-\Phi_{0}=0$, implies that $\Phi_{0}=0$ on $S$ and applying Lemma \ref{ch4lemma2} gives $\Psi_{0}=0$ on $S$. 
\end{proof}

\begin{theorem}
The operator \[\mathcal{M}^{12}:H^{1}(\Omega)\times\widetilde{H}^{-\frac{1}{2}}(S_{D})\times\widetilde{H}^{\frac{1}{2}}(S_{N})\longrightarrow H^{1}(\Omega)\times H^{\frac{1}{2}}(S),\] is invertible.
\end{theorem}

\begin{proof}
 Let $\mathcal{M}_{0}^{12}$ be the matrix operator defined by
\begin{equation*}\label{ch4M012}
  \mathcal{M}_{0}^{12}:=
  \left[ {\begin{array}{ccc}
   I & -V & W \\
   0 & -\mathcal{V} & \dfrac{1}{2}I \\
  \end{array} } \right].
\end{equation*}
The operator $\mathcal{M}_{0}^{12}$ is also bounded due to the mapping properties of the operators involved. Furthermore, the operator 
\begin{center}
$\mathcal{M}^{12}- \mathcal{M}_{0}^{12}: H^{1}(\Omega)\times\widetilde{H}^{-\frac{1}{2}}(S_{D})\times\widetilde{H}^{\frac{1}{2}}(S_{N})\longrightarrow H^{1}(\Omega)\times H^{\frac{1}{2}}(S)$
\end{center}
is compact due to the compact mapping properties of the operators $\mathcal{R}$ and $\mathcal{W}$, (cf. Theorem \ref{ch4corcompact} and Theorem \ref{ch4thcompVW}).

Let us prove that the operator $\mathcal{M}^{12}_{0}$ is invertible. For this purpose, we consider the following system with arbitrary right hand side $\widetilde{F}=[\widetilde{F_{1}},\widetilde{F_{2}}]^{\top}\in H^{1}(\Omega)\times H^{\frac{1}{2}}(S)$ and let $\mathcal{X}=(u,\psi,\phi)^{\top}\in H^{1}(\Omega)\times\widetilde{H}^{-\frac{1}{2}}(S_{D})\times\widetilde{H}^{\frac{1}{2}}(S_{N})$ be the vector of unknowns
\begin{equation}\label{ch4systemM120}
\mathcal{M}^{12}_{0}\mathcal{X}=\widetilde{F}.
\end{equation} 

Writing \eqref{ch4systemM120} component-wise,
\begin{subequations}
\begin{align}
 u- V\psi +W\phi &= \widetilde{F_{1}},\hspace{2em} in \hspace{0.5em}\Omega,\label{ch4i121}\\
 \frac{1}{2}\phi - \mathcal{V}\psi &= \widetilde{F_{2}}, \hspace{2em}on \hspace{0.5em}S.\label{ch4i122} 
\end{align}
\end{subequations}

Equation \eqref{ch4i122} restricted to $S_{D}$ gives:
\begin{equation}\label{ch4i123}
-r_{S_{D}}\mathcal{V}\psi = r_{S_{D}}\widetilde{F_{2}}.
\end{equation}

Due to the invertibility of the operator $\mathcal{V}$ (cf. Lemma \ref{ch4thinvV}), equation \eqref{ch4i123} is uniquely solvable on $S_{D}$.
Equation \eqref{ch4i123} means that $(\mathcal{V}\psi + \widetilde{F_{2}})\in \widetilde{H}^{\frac{1}{2}}(S_{N})$. Thus, the unique solvability of \eqref{ch4i123} implies that $\phi$ is also uniquely determined by the equation
\begin{equation*}
\phi=(2\mathcal{V}\psi +2\widetilde{F_{2}})\in \widetilde{H}^{\frac{1}{2}}(S_{N}).
\end{equation*}
Consequently, $u$ also is uniquely determined by the first equation \eqref{ch4i121} of the system. Furthermore, since $V\psi$, $W\phi\in H^{1}(\Omega)$, we have $u\in H^{1}(\Omega)$.

Thus, the operator $\mathcal{M}_{0}^{12}$ is invertible and the operator $\mathcal{M}^{12}$ is a zero index Fredholm operator due to the compactness of the operator $\mathcal{M}^{12}-\mathcal{M}_{0}^{12}$. Hence the Fredholm property and the injectivity of the operator $\mathcal{M}^{12}$, provided by item $iii)$ of Theorem  \ref{ch4remark}, imply the invertibility of operator $\mathcal{M}^{12}$.
\end{proof}

\section{Conclusions}
A new parametrix for the diffusion equation in non homogeneous media (with variable coefficient) has been analysed in this paper. Mapping properties of the corresponding parametrix based surface and volume potentials have been shown in corresponding Sobolev spaces. 

A BDIES for the original BVP has been obtained. Results of equivalence between the BDIES and the BVP has been shown along with the invertibility of the matrix operator defining the BDIES. 

Now, we have obtained an analogous system to the BDIES (M12) of \cite{mikhailov1} with a new family of parametrices which is uniquely solvable. Hence, further investigation about the numerical advantages of using one family of parametrices over another will follow.

Analogous results could be obtain for exterior domains following a similar approach as in \cite{exterior}. 

Further generalised results for Lipschitz domains can also be obtain by using the generalised canonical conormal derivative operator defined in \cite{traces, mikhailovlipschitz}. Moreover, these results can be generalised to Bessov spaces as in \cite{miksolandreg}.

\end{document}